\documentclass[a4paper,11pt]{article}
\usepackage[all]{xy}
\title{On o-minimal homotopy}

\usepackage{amsmath}
\usepackage{amscd}
\usepackage{amsfonts}
\usepackage{amssymb}
\usepackage{amsthm}
\usepackage{graphicx}
\usepackage{bm}
\newtheorem{deff}{Definition}[section]

\newtheorem{teo}[deff]{Theorem}
\newtheorem{lema}[deff]{Lemma}
\newtheorem{cor}[deff]{Corollary}

\newtheorem{fact}[deff]{Fact}
\theoremstyle{definition}

\newtheorem{obs}[deff]{Remark}

\newcommand{\V}{\textrm{Vert}}

\newcommand{\St}{\textrm{St}}

\newcommand{\coc}{\emph{co}}

\newcommand{\cat}{\textrm{cat}}
\newcommand{\catc}{cat}

\newcommand{\id}{\textrm{id}}
\begin{document}
\begin{center}ON THE O-MINIMAL LS-CATEGORY\end{center}
\begin{center}\begin{footnotesize}EL\'IAS BARO\footnote{Partially supported by MTM2008-00272/MTM and a FPU grant of MEC (Spain).

\emph{Email:} \texttt{elias.baro@uam.es}

\textit{Date}: April 28, 2009.


}\end{footnotesize}

\begin{scriptsize}Departamento de Matem\'aticas, Universidad Aut\'onoma de Madrid, 28049, Madrid, Spain.\end{scriptsize}
\end{center}
\begin{quote}
\begin{footnotesize}\begin{scriptsize}ABSTRACT\end{scriptsize}. We introduce the o-minimal LS-category of definable sets in o-minimal expansions of ordered fields and we establish a relation with the semialgebraic and the classical one. We also study the o-minimal LS-category of definable groups. Along the way, we show that two definably connected definably compact definable groups $G$ and $H$ are definable homotopy equivalent if and only if $\mathbb{L}(G)$ and $\mathbb{L}(H)$ are homotopy equivalent, where  $\mathbb{L}$ is the functor which associates to each definable group its corresponding Lie group via Pillay's conjecture.
\end{footnotesize} 
\end{quote}
\section{Introduction}
This paper is a contribution to the development of o-minimal topology, in particular to o-minimal homotopy. This development has found a successful approach through homology and cohomology theories. Recently, in \cite{07BO} the author and M. Otero also provide a homotopy theory to the o-minimal setting (see Fact \ref{t:compa}).

We introduce the o-minimal Lusternik-Schnirelmann category (in short o-minimal LS-category) of definable sets in an o-minimal expansion $\mathcal{R}$ of a real closed field. The classical one was originally introduced to provide a lower bound on the number of critical points for any smooth function on a manifold, then it became an important subject in algebraic topology. The LS-category of a topological space $X$, denoted by $\cat(X)$, is the least integer $m$ such that $X$ has an open cover of $m+1$ elements with each of them contractible to a point in $X$. The o-minimal LS-category of a definable set $X$, denoted by $\cat(X)^{\mathcal{R}}$, is defined in the obvious way (see Definition \ref{f:catinv}).

In Section \ref{s:LScat}, in analogy with the homotopy results in \cite{07BO}, we prove the following comparison result.
\begin{teo}\label{cor:catdefreal}Let $X$ be a semialgebraic set defined without parameters. Then $cat(X)^{\mathcal{R}}=cat(X(\mathbb{R}))$.
\end{teo}
On the other hand, recall that given a definably compact $d$-dimensional definable group $G$, the work of several authors (e.g. A. Berarducci, E. Hrushovski, Y. Peterzil, A. Pillay, M. Otero and others) in the positively resolution of Pillay's conjecture has shown that there exist a smallest type-definable subgroup $G^{00}$ of $G$ with bounded index such that $\mathbb{L}(G):=G/G^{00}$ with the logic topology is a compact $d$-dimensional Lie group (see \cite{05BOPP},\cite{08HPP},\cite{04P}). Moreover, if $G$ is definably connected then $\mathbb{L}(G)$ is connected. Recall that a subset of $\mathbb{L}(G)$ is closed in the logic topology if and only if its preimage by the projection $\pi:G\rightarrow \mathbb{L}(G)$ is a type-definable subset of $G$. With the logic topology $\pi:G\rightarrow \mathbb{L}(G)$ is a continuous epimorphism.

The main motivation to study the o-minimal LS-category is establish a topological analogy between a definably compact definably connected group $G$ and the connected compact Lie group $\mathbb{L}(G)$ associated to it. In this direction, it has been proved by A. Berarducci in \cite{07Be} that the cohomology groups of $G$ are isomorphic to those of $\mathbb{L}(G)$. Moreover, using the development of o-minimal homotopy mentioned above, A. Berarducci, M. Mamino and M. Otero prove in \cite{08pBMO} that the homotopy groups of $G$ and $\mathbb{L}(G)$ are isomorphic. The aim of Section \ref{s:LScatgr} is to prove that $\cat(G)^{\mathcal{R}}$ and $\cat(\mathbb{L}(G))$ are equal (see Corollary \ref{c:catgr}). To do this we establish the following stronger result.
\begin{teo}\label{teo1}Let $G$ be a definably connected definably compact definable group whose underlying set is a semialgebraic set defined without parameters. Then $G(\mathbb{R})$ is homotopy equivalent to $\mathbb{L}(G)$.
\end{teo}
As a consequence of Theorem \ref{teo1} we get the following.
\begin{cor}\label{c:isom}Let $G$ and $H$ be definably connected definably compact definable groups. Then $G$ and $H$ are definable homotopy equivalent if and only if $\mathbb{L}(G)$ and $\mathbb{L}(H)$ are homotopy equivalent.
\end{cor}

We point out that the preprint \cite{09pBM} by A. Berarducci and M. Mamino  has recently appeared with similar results to Theorem \ref{teo1} and Corollary \ref{c:isom}. However, the points of view of each paper are different. Here, we obtain the results via a transfer approach. In \cite{09pBM}, the development goes through a new homotopic study of the projection map $\pi:G\rightarrow \mathbb{L}(G)$.
%

The results of Section \ref{s:LScat} of this paper are part of the author's Ph.D. dissertation.\\
\textbf{Acknowledgements.} Part of the work for this paper was done while the author
was visiting Universidad de M\'alaga and he would like to thanks Professors Aniceto Murillo and Antonio Viruel for their hospitality. He also thanks Professor Margarita Otero for all the helpful discussions.
\section{Notation and preliminaries}
For the rest of the paper we fix an o-minimal expansion $\mathcal{R}$ of a real closed field $R$. We shall denote by $\mathcal{R}_{sa}$ the field structure of the real closed field $R$. We always take \emph{definable} to mean definable in $\mathcal{R}$ with parameters, except otherwise stated. We take the order topology on $R$ and the product topology on $R^n$ for $n>1$. All definable maps are assumed to be continuous. 
If a set $X$ is definable with parameters in some structure $\mathcal{M}$ we denote by $X(M)$ the realization of $X$ in $\mathcal{M}$.

Given a definable set $S$ and some definable subsets $S_1,\ldots,S_l$ of $S$ we say that $(K,\phi)$ is a \emph{triangulation} in $R^p$ of $S$ partitioning $S_1,\ldots,S_l$ if $K$ is a simplicial complex formed by a finite number of (open) simplices in $R^p$ and $\phi:|K|\rightarrow S$ is a definable homeomorphism, with $|K|=\bigcup_{\sigma\in K}\sigma \subset R^p$ the realization of $K$, such that each $S_i$ is the union of the images by $\phi$ of some simplices of $K$. We say that a simplicial complex $K'$ is a subdivision of a simplicial complex $K$ if each simplex of $K'$ is contained in a simplex of $K$ and  each simplex of $K$ equals the union of finitely many simplices of $K'$. We will use the standard notion  of barycentric subdivision of a simplicial complex (see \cite[Ch.8, \S 1.8]{98vD}).

Now, let us collect some results needed in the sequel. We start with a refinement of the triangulation theorem.
\begin{fact}[\textbf{Normal triangulation theorem}] \emph{\cite[Thm.1.4]{07pB}}\label{teo:triangulacion} Let $K$ be a simplicial complex and let $S_{1},\ldots,S_{l}$ be definable subsets of its realization $|K|$. Then, there is a subdivision $K'$ of $K$ and a definable homeomorphism $\phi':|K'|\rightarrow |K|$ such that\\
\emph{(i)} $(K',\phi')$ partitions all $S_{1},\ldots,S_{l}$ and each $\sigma\in K$,\\
\emph{(ii)} for every $\tau \in K'$ and $\sigma \in K$, if $\tau \subset \sigma$ then $\phi'(\tau)\subset \sigma$.\\
We say that $(K',\phi')$ is a \textbf{normal triangulation} of $|K|$ partitioning the subsets $S_1,\ldots,S_l$.
\end{fact}
\hspace{-0.6cm}Using the notation above, it follows from property (ii) that $\phi'$ is definably homotopic to $\id_{|K|}$ (see the proof of \cite[Thm.1.1]{07pB}). For this reason the normal triangulation theorem is a key tool to prove the following results concerning o-minimal homotopy already  mentioned in the introduction. Let $X$ and $Y$ be semialgebraic sets in $\mathcal{R}_{sa}$ and let $A$ and $B$ be semialgebraic subsets of $X$ and $Y$ respectively. The o-minimal homotopy set $[(X,A),(Y,B)]^{\mathcal{R}}$ is the collection of definable maps $f:X\rightarrow Y$ with $f(A)\subset B$ modulo definable homotopy mapping $A$ in $B$ (see \cite[\S 3]{07BO}).
\begin{fact}\label{t:compa}With the notation above, if $A$ is relatively closed in $X$ then,\\

\vspace{-0.4cm}
\hspace{-0.6cm}\emph{(a)}\emph{\cite[Cor.3.3]{07BO}} the map $[(X,A),(Y,B)]^{\mathcal{R}_{sa}} \rightarrow [(X,A),(Y,B)]^{\mathcal{R}}: [f] \mapsto [f]$ is a bijection,\\

\vspace{-0.4cm}
\hspace{-0.6cm}\emph{(b)}\emph{\cite[Ch.III, Thm.4.1]{85DK}} if $S$ a real closed field extension of $R$ then the map $[(X,A),(Y,B)]^{\mathcal{R}_{sa}} \rightarrow [(X(S),A(S)),(Y(S),B(S))]^{\mathcal{S}_{sa}}: [f] \mapsto [f(S)]$ is a bijection,\\

\vspace{-0.4cm}
\hspace{-0.6cm}\emph{(c)}\emph{\cite[Ch.III, Thm.5.1]{85DK}} if $R=\mathbb{R}$ then $[(X,A),(Y,B)]^{\mathcal{R}_{sa}} \rightarrow [(X,A),(Y,B)]: [f] \mapsto [f]$ is a bijection, where $[(X,A),(Y,B)]$ is the classical homotopy set.
\end{fact}
\hspace{-0.6cm}Moreover, if the homotopy sets under consideration are the homotopy groups of a semialgebraic pointed set $(X,x)$, i.e., $\pi_n(X,x)^{\mathcal{R}}=[(I^n,\partial I^n),(X,x)]^{\mathcal{R}}$, then the bijections are isomorphisms (see \cite[\S 4]{07BO}). We shall omit the superscript $\mathcal{R}$ if it is clear from the context.

We finish with the following interesting application of Fact \ref{t:compa} to the study of definable groups which will be used in Section \ref{s:LScatgr}.
\begin{fact}\emph{\cite[Thm.3.7]{08pBMO}}\label{BMO} Let $G$ be a definably connected definably compact group. Then $\pi_n(G)\cong \pi_n(\mathbb{L}(G))$ for all $n\geq 1$.
\end{fact}
\section{o-Minimal LS-category of definable sets}\label{s:LScat}
In this section we introduce the Lusternik-Schnirelmann category (in short LS-category) for definable sets. We apply the results of o-minimal homotopy from \cite{07BO} and the normal triangulation theorem to prove some comparison theorems concerning the LS-category. For a general reference on the classical Lusternik-Schnirelmann category see \cite{03CLOT}.
\begin{deff}Let $X$ be a definable set. We say that a definable subset $A$ of $X$ is \textbf{definably categorical} in $X$ if $A$ is definably contractible to a point in $X$ (not necessarily definably contractible in itself). We say that a definable cover $\{ V_i\}_{i=1}^{m}$ of $X$ is a \textbf{definable categorical cover} of $X$ if each $V_i$ is definably categorical  in $X$.
\end{deff}

\begin{fact}\emph{\cite[Lem.1.29]{03CLOT}}\label{cathomot} Let $X$ and $Y$ be definable sets. Let $f:X\rightarrow Y$ and $g:Y\rightarrow X$ be definable maps such that $g\circ f \sim id_X$. Then $f^{-1}(V)$ is a definable categorical subset of $X$ for each definable categorical subset $V$ of $Y$.\end{fact}
\begin{proof}Let $F:X\times I\rightarrow X$  be a definable homotopy from $\id_X$ to $g\circ f$. Let $y_0\in Y$ and let $H:V\times I\rightarrow Y$ be a definable homotopy  such that $H(y,0)=y$ for all $y\in V$ and $H(y,1)=y_0$ for all $y\in V$. Denote by $U=f^{-1}(V)$ and consider the definable map $G:U\times I\rightarrow X$ defined by 
\begin{displaymath}G(x,t)=\left\{
\begin{array}{l}
F(x,2t) \text{  \  \ for all  \ \ } (x,t)\in U\times [0,\frac{1}{2}],\\
g(H(f(x),2t-1)) \text{  \  \ for all  \ \ } (x,t)\in U\times [\frac{1}{2},1].\\
\end{array}\right.
\end{displaymath}
Note that $G(x,0)=x$ and $G(x,1)=g(x_0)$ for all $x\in U$, i.e., $U$ is a definable categorical subset of $X$.\end{proof}

Every definable set $X$ has a definable categorical open cover. Indeed, by the triangulation theorem and Fact \ref{cathomot} we can assume that $X=|K|$ for a simplicial complex $K$. We denote by $K'$ the first barycentric subdivision of $K$. By \cite[Prop III.1.6]{85DK}, the open definable subset $\St_{K'}(v)$ of $|K|$ is definably categorical in $|K|$ for each $v\in \V(K)\cap |K|$. Therefore, $\{\St_{K'}(v):v\in \V(K)\cap |K|\}$ is a finite definable categorical open cover of $|K|$. Alternatively, every definable set $X$ is a finite union of open cells (which are already definably contractible).
\begin{deff}The \textbf{o-minimal LS-category} of a definable set $X$, denoted by $\catc(X)^{\mathcal{R}}$, is the least integer $m$ such that $X$ has a definable categorical open cover of $m+1$ elements. 
\end{deff}
For example, by definition we have that a definable set $X$ is definably contractible if and only if $\cat(X)^{\mathcal{R}}=0$. Now, we prove that the o-minimal LS-category is homotopy invariant.
\begin{fact}\emph{\cite[Lem.1.30]{03CLOT}}\label{f:catinv} Let $X$ and $Y$ be definable homotopy equivalent definable sets. Then $\catc(X)^{\mathcal{R}}=\catc(Y)^{\mathcal{R}}$.
\end{fact}
\begin{proof}Let $f:X\rightarrow Y$ and $g:Y\rightarrow X$ be definable maps such that $f\circ g\thicksim \id_{Y}$ and $g\circ f\thicksim \id_{X}$. We show that $\cat(Y)^{\mathcal{R}}\leq \cat(X)^{\mathcal{R}}$, the other inequality by symmetry. Let $\{U_i\}^{m+1}_{i=1}$ be a definable categorical open cover of $X$. By Fact \ref{cathomot}, $g^{-1}(U_i)$ is definably categorical in $Y$ for each $i=1,\ldots,m+1$. Hence $\{g^{-1}(U_i)\}^{m+1}_{i=1}$ is a categorical open cover of $Y$.
\end{proof}
\begin{teo}\label{t:catcomp1}Let $X\subset R^n$ be a semialgebraic set. Then $\catc(X)^{\mathcal{R}_{sa}}=\catc(X)^{\mathcal{R}}$.
\end{teo}
\begin{proof}Clearly, $\cat(X)^{\mathcal{R}_{sa}}\geq \cat(X)^{\mathcal{R}}$. We show that $\cat(X)^{\mathcal{R}_{sa}}\leq \cat(X)^{\mathcal{R}}$. By the Triangulation theorem we can assume that $X=|K|$ for some simplicial complex $K$. Let $m=\cat(X)^{\mathcal{R}}$ and let $U_1,\ldots,U_{m+1}$ be a definable categorical open cover of $|K|$. By the normal triangulation theorem \ref{teo:triangulacion} there is a subdivision $K'$ of $K$ and a definable homeomorphism $\phi:|K'|\rightarrow |K|$ such that $(K',\phi)$ partitions $U_1,\ldots,U_{m+1}$. Note that since $K'$ is a subdivision of $K$ we have $|K'|=|K|$ (this is the reason why we use the normal triangulation theorem instead of the standard one). Therefore the open subset $V_i:=\phi^{-1}(U_i)$ is the realization of a subcomplex of $K'$ and hence a semialgebraic subset of $|K'|$ for each $i=1,\ldots,m+1$. Since $\phi$ is a definable homeomorphism, $V_i$ is a definable categorical subset of $|K|$ for all $i=1,\ldots,m+1$ (see Fact \ref{cathomot}). Now, by Fact \ref{t:compa}.(a) if a semialgebraic subset $A$ of a semialgebraic set $B$ is definably contractible in $B$, then it is semialgebraically contractible in $B$. Hence, $V_i$ is a semialgebraic categorical subset of $|K|$ for all $i=1,\ldots,m+1$. Then $\{V_i\}_{i=1}^{m+1}$ is a semialgebraic categorical open cover of $|K|$ and hence $\cat(X)^{\mathcal{R}_{sa}}\leq m$.
\end{proof}
\begin{teo}\label{t:catcomp2}Let $X\subset R^n$ be a semialgebraic set. Let $S$ be a real closed field extension of $R$. Then $\catc(X)^{\mathcal{R}_{sa}}=\catc(X(S))^{\mathcal{S}_{sa}}$. 
\end{teo}
\begin{proof}It is immediate that $\cat(X)^{\mathcal{R}_{sa}}\geq\cat(X(S))^{\mathcal{S}_{sa}}$. For, given a semialgebraic categorical open cover $\{U_i\}_{i=1}^{m+1}$ of $X$, $\{U_i(S)\}_{i=1}^{m+1}$ is clearly a semialgebraic categorical open cover of $X(S)$. We show that $\cat(X)^{\mathcal{R}_{sa}}\leq\cat(X(S))^{\mathcal{S}_{sa}}$. Let $\{V_i\}_{i=1}^{m+1}$ be a semialgebraic categorical open cover of $X(S)$. Let $$H_i:V_i\times I\rightarrow X(S)$$ be a semialgebraic homotopy such that $H_i(x,0)=x$ for all $x\in V_i$ and $H_i(x,1)=x_i\in X(S)$ for all $x\in V_i$ for each $i=1,\ldots,m+1$. Without loss of generality we can assume that $x_i\in R^n$. Indeed, let $x'_i\in X(S)\cap R^n$ be a point which lies in the same semialgebraic connected component of $x_i$. Consider a semialgebraic curve $\alpha_i:I\rightarrow X(S)$ such that $\alpha_i(0)=x_i$ and $\alpha_i(1)=x'_i$. Then, we can replace $H_i$ by the semialgebraic homotopy $H'_i:V_i\times I\rightarrow X(S)$ with $H_i'(x,t)=H_i(x,2t)$ for all $(x,t)\in V_i\times [0,\frac{1}{2}]$ and $H_i'(x,t)=\alpha_i(2t-1)$ for all $(x,t)\in V_i\times [\frac{1}{2},1]$. 

On the other hand, by the triangulation theorem and Fact \ref{f:catinv} we can assume that $X(S)=|K|$ for a simplicial complex $K$ in $S$ whose vertices lie in $R$ as well as the $V_i$'s are realizations of subcomplexes of $K$ semialgebraically categorical in $|K|$ (possibly with parameters from $S$). Furthermore, by Fact \ref{t:compa}.(b) the $V_i$'s are also semialgebraically categorical in $|K|$ with parameters from $R$ and therefore $\{V_{i}(R)\}_{i=1}^{m+1}$ is a semialgebraic categorical open cover of $X$. Hence, $\cat(X)^{\mathcal{R}_{sa}}\leq\cat(X(S))^{\mathcal{S}_{sa}}$, as required.
\end{proof}
The following fact allows us to work with closed simplicial complexes in the proof of Theorem \ref{t:catcomp3}.
\begin{fact}\emph{\cite[Ch.III, Prop.1.6, 1.8]{85DK}}\label{f:core} Let $K$ be the first barycentric subdivision of some simplicial complex. Let $co(K)$ be the closed simplicial subcomplex of $K$ consisting in all simplexes of $K$ whose faces are also simplexes of $K$. Then there is a semialgebraic retraction $r:|K|\rightarrow |\coc(K)|$ such that $(1-t)x+t\cdot r(x)\in |K|$ for all $(x,t)\in |K|\times I$ and hence  $H:|K|\times I\rightarrow |K|:(x,t)\mapsto (1-t)x+t\cdot r(x)$ is a canonical semialgebraic strong deformation retraction.
\end{fact}
\begin{teo}\label{t:catcomp3}Let $X\subset \mathbb{R}^n$ be a semialgebraic set. Then $\catc(X)^{\mathbb{R}_{sa}}=\catc(X)$, where $\catc(X)$ denotes the classical LS-category of $X$.
\end{teo}
\begin{proof}Clearly, $\cat(X)^{\mathbb{R}_{sa}}\geq \cat(X)$. We show that $\cat(X)^{\mathbb{R}_{sa}}\leq \cat(X)$. We can assume that $X=|K|$ for some simplicial complex $K$. Moreover, since strong deformation retracts are homotopy equivalences, by Fact \ref{f:core} and Fact \ref{f:catinv} we can assume that $K$ is closed. Let $\{U_i\}_{i=1}^{m+1}$ be a categorical open cover of $|K|$. We will construct a semialgebraic categorical open cover $\{V_i\}_{i=1}^{m+1}$ of $|K|$. Firstly, by the shrinking lemma we can assume that each $\overline{U_i}$ is contractible in $|K|$. Furthermore, by the Lebesgue's number lemma we can also assume that for each $\sigma\in K$ there is $i\in \{1,\ldots,m+1\}$ such that $\sigma\subset U_i$. We define $\mathcal{F}_i:=\{\sigma\in K:\sigma\subset U_i\}$ and $$A_i:=\bigcup_{\sigma\in \mathcal{F}_i}\sigma$$ for each $i=1,\ldots,m+1$. Note that (i) $|K|=A_1\cup \cdots \cup A_{m+1}$ and (ii) each $\overline{A_i}$ is contractible in $|K|$. On the other hand, each $\overline{A_i}$ is a semialgebraic strong deformation retract of the open semialgebraic set $V_i:=\St_{K'}(\overline{A_i})$, where $K'$ is the first barycentric subdivision of $K$ (see \cite[Prop III.1.6]{85DK}). Therefore, by (ii), $V_i$ is (not necessarily semialgebraically) contractible in $|K|$. Now, by Fact \ref{t:compa}.(c) if a semialgebraic subset $A$ of a semialgebraic set $B$ is contractible in $B$, then it is semialgebraically contractible in $B$. Hence, each $V_i$ is semialgebraically contractible in $|K|$ and hence, by (i), $\{V_i\}_{i=1}^{m+1}$ is a semialgebraic categorical open cover of $|K|$. We deduce that $\cat(X)^{\mathbb{R}_{sa}}\leq\cat(X)$, as required.
\end{proof}
\begin{cor}\label{cor:catinvar}The o-minimal LS-category is invariant under elementary extensions and o-minimal expansions.
\end{cor}
\begin{proof}This follows from Theorem \ref{t:catcomp1} and Theorem \ref{t:catcomp2}. 
\end{proof}
\begin{proof}[Proof of Theorem \ref{cor:catdefreal}]We denote by $\overline{\mathbb{Q}}$ the real algebraic numbers. It follows from Corollary \ref{cor:catinvar} and Theorem \ref{t:catcomp3} that $\cat(X)^{\mathcal{R}}=\cat(X)^{\mathcal{R}_{sa}}=\cat(X(\overline{\mathbb{Q}}))^{\overline{\mathbb{Q}}_{sa}}=\cat(X(\mathbb{R}))^{\mathbb{R}_{sa}}=\cat(X(\mathbb{R}))$.\end{proof}
We can apply the comparison Theorem \ref{cor:catdefreal} to transfer classical results concerning the LS-category to the o-minimal setting.
\begin{cor}\label{c:catdim}Let $X$ be a definably connected definable set. Then $$\cat(X)^{\mathcal{R}}\leq dim(X).$$
\end{cor}
\begin{proof}By the Triangulation theorem, Fact \ref{f:core} and \ref{f:catinv}, we can assume that $X=|K|$ for a closed simplicial complex $K$ whose vertices lie in the real algebraic numbers $\overline{\mathbb{Q}}$. Now, it follows from Theorem \ref{cor:catdefreal} that $\cat(|K|(\mathbb{R}))^{\mathcal{R}}=\cat(|K|(\mathbb{R}))$. By the classical version of Corollary \ref{c:catdim} (see \cite[Thm. 1.7]{03CLOT}), $$\cat(|K|(\mathbb{R}))\leq \textrm{dim}(|K|(\mathbb{R}))^{\text{top}},$$ where $\dim(|K|(\mathbb{R}))^{\text{top}}$ denotes the covering dimension of $|K|(\mathbb{R})$. On the other hand, since $K$ is a simplicial complex, $\dim(|K|(\mathbb{R}))^{\text{top}}$ is exactly the dimension of $K$ as a simplicial complex, i.e., $\dim(|K|(\mathbb{R}))^{\text{top}}=\textrm{dim}(|K|)$, the latter being also the o-minimal dimension, as required.
\end{proof}
\begin{cor}\label{c:cor3}Let $X$ be definable set and let $n\geq 1$ such that $\pi_r(X)^{\mathcal{R}}=0$ for all $r=0,\ldots,n-1$. Then $cat(X)^{\mathcal{R}}\leq dim(X)/n$.
\end{cor}
\begin{proof}Similar to the proof of Corollary \ref{c:catdim} using the corresponding classical statement to that of Corollary \ref{c:cor3} (see \cite[Thm. 1.50]{03CLOT}).
\end{proof}
\begin{cor}\label{c:cor2}Let $X$ be a definable set. Let cuplength$_{\mathbb{Q}}(X)^{\mathcal{R}}$ be the least integer $k$ such that all $(k+1)$-fold cup products vanish in the reduced cohomology $\tilde{H}^{*}(X;\mathbb{Q})^{\mathcal{R}}$. Then $\catc(X)^{\mathcal{R}}\geq \textrm{cuplength}_{\mathbb{Q}}(X)^{\mathcal{R}}$.
\end{cor}
\begin{proof}As before, this follows from Theorem \ref{cor:catdefreal}, the o-minimal cohomological theory developed in \cite{04EO} and the corresponding classical statement to that of Corollary \ref{c:cor2} (see \cite[Thm. 1.5]{03CLOT}). 
\end{proof}
\section{Homotopy types of definable groups}\label{s:LScatgr}
In this section we assume that $\mathcal{R}$ is sufficiently saturated. In accordance with the notation used in the introduction, we shall denote by $\mathbb{L}$ the functor

\vspace{-0.2cm}
\begin{small}
$$\begin{array}{rcl}\mathbb{L}:\{\text{Definably compact definable groups}\}& \rightarrow & \{\text{Compact Lie groups}\}\\
G & \mapsto & G/G^{00}.
\end{array}$$\end{small}
\hspace{-0.1cm}which maps definable homomorphisms to continuous homomorphisms (see \cite[Thm.5.2]{07Be}). We collect here some properties of the functor $\mathbb{L}$ studied in \cite{07Be} which will be used in the sequel without mention.
\begin{fact}\label{connec}Let $G$ and $H$ be definably compact definable groups and let $\pi:G\rightarrow \mathbb{L}(G)$ be the projection.\\
\emph{(i)} The functor $\mathbb{L}$ is exact.\\
\emph{(ii)} If $H$ is a definable subgroup of $G$ then $\pi(H)=\mathbb{L}(H)$.\\
\emph{(iii)} $\mathbb{L}(G^0)=\mathbb{L}(G)^0$.\\
\emph{(iv)} $\mathbb{L}(G\times H)=\mathbb{L}(G)\times \mathbb{L}(H)$.
\end{fact}
\begin{proof}\emph{(i)} and \emph{(ii)} can be found in \cite[Thm.5.2]{07Be} and \cite[Thm.4.4]{07Be} respectively. \emph{(iii)} follows from \emph{(i)} and \emph{(iii)} follows from \cite[Cor.4.7]{07Be}.
\end{proof}
As we pointed out in the introduction, the aim of this section is to prove Theorem \ref{teo1}. To do this we first prove the following. Recall that given a definable group $G$ the commutator subgroup $G':=[G, G]$ might not be definable. However, if $G$ is definably compact definably connected then $G'$ is a definably connected definable subgroup of $G$ (see below and \cite[Cor.6.4]{HPP08}).
\begin{teo}\label{teo2}Let $G$ be a definably connected definably compact definable group. Then $Z(G)^0\times G'$ is definable homotopy equivalent to $G$.
\end{teo}
The latter is motivated by the following classical result  concerning compact Lie groups proved by A. Borel. We include the proof for completeness.
\begin{fact}\emph{\cite[Prop.3.1]{B}}\label{borel} Let $G$ be a compact, connected Lie group. Then $G$  is homeomorphic to the topological direct product of its commutator subgroup $G'$ and the connected component  $Z(G)^0$ of its center.
\end{fact}
\begin{proof}Firstly, note that $G=Z(G)^0G'$ and $Z(G)^0\cap G'$ is finite. Let $n$ be the dimension of $Z(G)^0$. Consider a connected $(n-1)$-dimensional Lie subgroup $Z_1$ of the torus $Z(G)^0$. Then, the projection $G \rightarrow G/Z_1G'$ is a principal bundle with a circle as its base and with the connected Lie group $Z_1G'$ as its fiber. On the other hand, the equivalence classes of principal bundles over an sphere $S^n$ with an arcwise connected group $Z_1G'$ as its fiber is in 1-1 correspondence with $\pi_{n-1}(Z_1G')$ (see \cite[Cor.18.6]{51S}). Hence, since  $\pi_0(Z_1G')=0$, the bundle $G \rightarrow G/Z_1G'$ is equivalent to the trivial one, so that $G$ is homeomorphic to $G/Z_1G'\times Z_1G'$. By recurrence we get that $G$ is homeomorphic to $Z(G)^0\times G'$.
\end{proof}
\begin{obs}The proof of Fact \ref{borel} does not apply in the o-minimal setting. For, even though we could prove an o-minimal version of the bundle classification used above, there are definably connected definably compact abelian definable groups without any proper infinite definable subgroup (see \cite[\S 5]{99PS}). Anyway, we conjecture that every definably connected definably compact definable group $G$ is definably homeomorphic to $Z(G)^0\times G'$.
\end{obs}
On the other hand, to prove Theorem \ref{teo2} we will use the following structural result about definably compact groups recently proved by E. Hrushovski, A. Pillay and Y. Peterzil. Recall that a definable group $G$ is semisimple if and only if it has no infinite abelian normal (definable) subgroup.
\begin{fact}\emph{\cite[Cor.6.4]{HPP08}}\label{HPP} Let $G$ be a definably connected definably compact group. Then $G':=[G,G]$ is definable, definably connected and semisimple. Moreover, the map $p:Z(G)^0\times G'\rightarrow G:(g,h)\mapsto gh$ is a homomorphism with finite kernel.
\end{fact}
The following is an immediate consequence of Fact \ref{HPP}.
\begin{lema}\label{l:cender}Let $G$ be a definably connected, definably compact definable group. Then $\mathbb{L}(G')=\mathbb{L}(G)'$ and $\mathbb{L}(Z(G)^0)=Z(\mathbb{L}(G))^0$. 
\end{lema}
\begin{proof}The projection $\pi:G\rightarrow \mathbb{L}(G)$ is a surjective homomorphism. Hence, $\pi(G')=\pi(G)'$. Since $\pi(G')=\mathbb{L}(G')$, we deduce that $\mathbb{L}(G')=\mathbb{L}(G)'$. Now, we show that $\mathbb{L}(Z(G)^0)=Z(\mathbb{L}(G))^0$. Since $\pi$ is a surjective homomorphism, $\mathbb{L}(Z(G))=\pi(Z(G))< Z(\pi(G))=Z(\mathbb{L}(G))$. Then, $$\mathbb{L}(Z(G)^0)=\mathbb{L}(Z(G))^0<Z(\mathbb{L}(G))^0.$$ Hence, it is enough to prove that $\dim(\mathbb{L}(Z(G)^0))=\dim(Z(\mathbb{L}(G))^0)$. By Fact \ref{HPP}, $\dim(G)=\dim(Z(G)^0\times G')=\dim(Z(G)^0)+\dim(G')$ and hence we have $\dim(\mathbb{L}(G))=\dim(\mathbb{L}(Z(G)^0))+\dim(\mathbb{L}(G'))$. On the other hand, it follows from Fact \ref{borel} that $\dim(\mathbb{L}(G))=\dim(Z(\mathbb{L}(G))^0)+\dim(\mathbb{L}(G)')$. Since $\dim(\mathbb{L}(G'))=\dim(\mathbb{L}(G)')$, we deduce that $\dim(\mathbb{L}(Z(G)^0))=\dim(Z(\mathbb{L}(G))^0)$, as required.
\end{proof}
\begin{cor}\label{c:grhom}Let $G$ be a definably connected, definably compact definable group. Then $\pi_n(G)\cong \pi_n(Z(G)^0\times G')$ for all $n\geq 1$.
\end{cor}
\begin{proof}By Lemma \ref{l:cender}, Fact \ref{BMO} and Fact \ref{borel} we have the following isomorphisms
$$\begin{array}{rcl}
\pi_n(G) & \cong  & \pi_n(\mathbb{L}(G))\cong \pi_n(Z(\mathbb{L}(G))^0 \times \mathbb{L}(G)')\cong\\
& \cong & \pi_n(Z(\mathbb{L}(G))^0) \times \pi_n(\mathbb{L}(G)')\cong\\
& \cong & \pi_n(\mathbb{L}(Z(G)^0)) \times \pi_n(\mathbb{L}(G'))\cong \\
& \cong & \pi_n(Z(G)^0) \times \pi_n(G')\cong \\
 & \cong & \pi_n(Z(G)^0 \times G'),
\end{array}$$
as required.
\end{proof}
Note that the definable homomorphism $p:Z(G)^0\times G'\rightarrow G:(g,h)\mapsto gh$ from Fact \ref{HPP} is a definable homotopy equivalence if and only if $p$ is an isomorphism. Indeed, if $p$ is a definable homotopy equivalence then $p_*:\pi_1(Z(G)^0\times G')\rightarrow \pi_1(G')$ is an isomorphism. On the other hand, since $p$ is a surjective homomorphism with finite kernel we have that $p$ is a definable covering (see \cite[Prop.2.11]{04EO}). Hence, it follows from \cite[Prop.2.9]{04EO} that the cardinality of $p^{-1}(e)$ equals the one of $\pi_1(G')/p_*(\pi_1(Z(G)^0\times G'))$, so that $p^{-1}(e)$ is trivial. Then, $p$ is an isomorphism, as required.

Even though $p$ above is not necessarily a definable homotopy equivalence, we now prove that $Z(G)^0\times G'$ and $G$ are actually definable homotopy equivalent.
\begin{proof}[Proof of Theorem \ref{teo2}] Let $d=\dim(Z(G)^0)$. Note that by Lemma \ref{l:cender}, we also have  $d=\dim(Z(\mathbb{L}(G))^0)$. It suffices to prove that $G$ is definable homotopy equivalent to $$\mathbb{T}_R^d\times G',$$ where $\mathbb{T}_R^d$ is the $d$-dimensional torus defined as the subset $[0,1)^d$ of $R^d$ with the sum operation modulo $1$. Indeed, $Z(G)^0$ is definable homotopy equivalent to $\mathbb{T}_R^d$ (see \cite[Cor.4.4]{08pBMO}).

Firstly, we show that $\pi_1(G)\cong \mathbb{Z}^d\times \text{Tor}(\pi_1(G))$. For, it follows from the proof of Corollary \ref{c:grhom} that $\pi_1(G)\cong \pi_1(Z(\mathbb{L}(G))^0)\times \pi_1(\mathbb{L}(G)')\cong \mathbb{Z}^d \times \pi_1(\mathbb{L}(G)')$. Moreover, $\mathbb{L}(G)'$ is a semisimple compact Lie group and hence $\pi_1(\mathbb{L}(G)')$ is finite, so that $\pi_1(\mathbb{L}(G)')\cong \text{Tor}(\pi_1(G))$, as required. In particular, since $\pi_1(G')\cong \pi_1(\mathbb{L}(G)')$ (see Fact \ref{BMO}), we have proved that $\pi_1(G')$ and $\text{Tor}(\pi_1(G))$ are isomorphic finite groups.

Now, take $\gamma_1,\ldots,\gamma_d:I\rightarrow G$ definable curves such that $$[\gamma_1]+\text{Tor}(\pi_1(G)),\ldots,[\gamma_d]+\text{Tor}(\pi_1(G)),$$ freely generate the group $\pi_1(G)/\text{Tor}(\pi_1(G))(\cong \mathbb{Z}^d)$. Consider the definable map,
$$f:\mathbb{T}_R^d\times G' \rightarrow G:(t_1,\ldots,t_d,g)\mapsto \gamma_1(t_1)\cdots\gamma_d(t_d)g.$$
We show that $f$ is a definable homotopy equivalence. By the o-minimal Whitehead theorem (see \cite[Thm.5.6]{07BO}) it suffices to prove that the homomorphism $f_*:\pi_n(\mathbb{T}_R^d\times G')\rightarrow \pi_n(G)$ is actually an isomorphism for all $n\geq 1$. Consider the definable maps $i:G'\rightarrow G:g\mapsto g$ and $j:\mathbb{T}_R^d\rightarrow G:(t_1,\ldots,t_d)\mapsto \gamma_1(t_1)\cdots\gamma_d(t_d)$. Since $G$ is a definable group, we can regard $f_*$ as the homomorphism
$$f_*:\pi_n(\mathbb{T}_R^d)\times \pi_n(G')\rightarrow \pi_n(G):(x,y)\mapsto j_*(x)+i_*(y).$$

\hspace{-0.6cm}{\em Claim: The homomorphism $i_*:\pi_n(G')\rightarrow \pi_n(G)$ is an isomorphism for every $n\geq2$ and injective for $n=1$.}\\

\vspace{-0.3cm}
\hspace{-0.6cm}Granted the claim and since $\pi_n(\mathbb{T}_R^d)=\pi_n(\mathbb{T}_\mathbb{R}^d)=0$ for all $n\geq2$ (see Fact \ref{t:compa}), we deduce that $f_*=i_*$ is an isomorphism for all $n\geq2$. To finish the proof we have to show that $f_*$ is an isomorphism for $n=1$. Firstly, note that by definition of the $\gamma_i$'s we have that the $[\gamma_i]$'s are $\mathbb{Z}-$linear independent and hence $j_*$ is injective. In particular, $j_*(\pi_1(\mathbb{T}_R^d))$ is torsion free. On the other hand, it follows from the claim that $i_*$ is injective. Hence, $i_*(\pi_1(G'))$ is a finite subgroup of $\pi_1(G)$ with the cardinality of $\text{Tor}(\pi_1(G))$, so that $i_*(\pi_1(G'))=\text{Tor}(\pi_1(G))$. Therefore, $j_*(\pi_1(\mathbb{T}_R^d))\cap i_*(\pi_1(G'))$ is trivial. We deduce that $f_*$ is injective. Finally, by definition of the definable curves $\gamma_i$'s we have that $\text{Im}(f_*)/\text{Tor}(\pi_1(G))=\pi_1(G)/\text{Tor}(\pi_1(G))$ and hence $\text{Im}(f_*)=\pi_1(G)$, as required.

\hspace{-0.6cm}\emph{Proof of the Claim.} By Fact \ref{HPP}, the kernel of $p:Z(G)^0\times G'\rightarrow G:(h,g)\mapsto h\cdot g$ is finite and therefore $p$ is a definable covering (see \cite[Prop.2.11]{04EO}). Then $p$ is a definable fibration (see  \cite[Thm.4.10]{07BO}) and hence, by the definable fibration property (see \cite[Cor.4.11]{07BO}), we have that $p_*$ is an isomorphism for all $n\geq 2$ and injective for $n=1$. Since $Z(G)^0$ is a definably compact abelian definable group, $\pi_n(Z(G)^0)=0$ for all $n\geq 2$ (see \cite[Cor.3.3]{08pBMO}). We deduce that $i_*=p_*$ is an isomorphism for all $n\geq 2$. On the other hand, the inclusion $k:G'\rightarrow Z(G)^0\times G'$ induces and injective map $k_*:\pi_1(G')\rightarrow \pi_1(Z(G)^0\times G')$ and hence $i_*=p_*\circ k_*$ is injective for $n=1$.
\end{proof}
As a consequence of Theorem \ref{teo2}, we obtain the following general result (compare with the discussion preceding the proof of Theorem \ref{teo2} above).
\begin{cor}Let $H$ and $G$ be definably connected definably compact definable groups such that $\pi_1(G)\cong \pi_1(H)$. If there exists a surjective homomorphism $p:G\rightarrow H$ with finite kernel then $G$ and $H$ are definable homotopy equivalent.
\end{cor}
\begin{proof}Firstly, it follows from Corollary \ref{c:grhom} that $\pi_1(G)\cong \mathbb{Z}^n\times \pi_1(G')$ and $\pi_1(H)\cong \mathbb{Z}^m\times \pi_1(H')$, where $n=\dim(Z(G)^0)$ and $m=\dim(Z(H)^0)$. Moreover, since $G'$ and $H'$ are definably compact and semisimple, both $\pi_1(G')$ and $\pi_1(H')$ are finite. Now, since $\pi_1(G)\cong \pi_1(H)$, we deduce that $n=m$ and the cardinality of $\pi_1(G')$ and $\pi_1(H')$ are equal.

By the proof of Theorem \ref{teo2} we have that $G$ and $H$ are definably homotopy equivalent to $\mathbb{T}^n\times G'$ and $\mathbb{T}^m\times H'$ respectively. Since $n=m$, it is enough to prove that $G'$ and $H'$ are definable homotopy equivalent. Actually, we show that $p|_{G'}:G'\rightarrow H'$ is an isomorphism. Since $p|_{G'}$ is a surjective homomorphism with finite kernel we have that $p|_{G'}$ is a definable covering homomorphism. In particular, $(p|_{G'})_*:\pi_1(G')\rightarrow \pi_1(H')$ is injective  (see \cite[Cor.2.8]{04EO}). Moreover, since $\pi_1(G')$ and $\pi_1(H')$ are finite groups with the same cardinality, we deduce that $(p|_{G'})_*(\pi_1(G'))=\pi_1(H')$. On the other hand, the cardinality of $(p|_{G'})^{-1}(e)$ equals the one of $\pi_1(H')/ (p|_{G'})_*(\pi_1(G'))$, so that $p|_{G'}$ is an isomorphism.
\end{proof}
We are now ready to prove Theorem \ref{teo1}.
\begin{proof}[Proof of the Theorem \ref{teo1}]By Fact \ref{HPP}, $G'$ is a definably connected definably compact semisimple definable group and hence, by a result of  M. Edmundo, G. Jones and N. Peatfield, $G'$ has a very good reduction, i.e., there is a semialgebraic group $H$ defined without parameters such that $G$ is definably isomorphic to $H$ (see, e.g., \cite[Thm.4.4]{HPP08}). By Lemma \ref{l:cender}, we have that $d:=\dim(Z(G)^0)=\dim(Z(\mathbb{L}(G))^0)$. We show that both $G(\mathbb{R})$ and $\mathbb{L}(G)$ are homotopy equivalent to $$\mathbb{T}^d_{\mathbb{R}}\times H(\mathbb{R}),$$
where $\mathbb{T}_{\mathbb{R}}^d$ is the $d$-dimensional torus defined as the subset $[0,1)^d$ of $\mathbb{R}^d$ with the sum operation modulo $1$. First, we prove that $\mathbb{L}(G)$ is homotopy equivalent to $\mathbb{T}^d_{\mathbb{R}}\times H(\mathbb{R})$. For, by Lemma \ref{l:cender} and \cite[Thm.1.6]{07Be} we have $\mathbb{L}(G)'= \mathbb{L}(G') \cong H(\mathbb{R})$ and hence $Z(\mathbb{L}(G))^{0}\times \mathbb{L}(G)'$ is isomorphic to $\mathbb{T}^{d}_{\mathbb{R}}\times H(\mathbb{R})$. Then, by Fact \ref{borel},  $\mathbb{L}(G)$ is homotopy equivalent (actually homeomorphic) to $\mathbb{T}^{d}_{\mathbb{R}}\times H(\mathbb{R})$. On the other hand, by the proof of Theorem \ref{teo2}, $G$ is definable homotopy equivalent to $\mathbb{T}_{R}^{d}\times H$. Now, since both $G$ and $\mathbb{T}_{R}^{d}\times H$ are semialgebraic sets defined without parameters, it follows from Fact \ref{t:compa} that $G$ is semialgebraic homotopy equivalent to $\mathbb{T}_{R}^{d}\times H$ without parameters. In particular, $G(\mathbb{R})$ is semialgebraic homotopy equivalent to $\mathbb{T}^{d}_{\mathbb{R}}\times H(\mathbb{R})$ without parameters, as required.
\end{proof}
\begin{proof}[Proof of Corollary \ref{c:isom}] By the triangulation theorem we can assume that the underlying sets of both $G$ and $H$ are semialgebraic without parameters. By Theorem \ref{teo1}, $\mathbb{L}(G)$ and $\mathbb{L}(H)$ are homotopy equivalent to $G(\mathbb{R})$ and $H(\mathbb{R})$ respectively. Now, if $G$ and $H$ are definable homotopy equivalent then $G$ and $H$ are semialgebraic homotopy equivalent without parameters (see Fact \ref{t:compa}). Hence  $G(\mathbb{R})$ and $H(\mathbb{R})$ are (semialgebraic) homotopy equivalent (without parameters), so that $\mathbb{L}(G)$ and $\mathbb{L}(H)$ are homotopy equivalent. On the other hand, if $\mathbb{L}(G)$ and $\mathbb{L}(H)$ are homotopy equivalent then $G(\mathbb{R})$ and $H(\mathbb{R})$ are homotopy equivalent. Hence, by Fact \ref{t:compa}.(c) we have that $G(\mathbb{R})$ and $H(\mathbb{R})$ are semialgebraic homotopy equivalent without parameters, so that $G$ and $H$ are semialgebraic homotopy equivalent (without parameters).
\end{proof}
\begin{obs}Theorem \ref{teo1} allows us to extend the functor $\mathbb{L}$ to definable maps up to homotopy. That is, given two definably connected definably compact definable groups $G$ and $H$, consider the o-minimal homotopy set $[G,H]$ and the homotopy set $[\mathbb{L}(G),\mathbb{L}(H)]$ (see the definition after Fact \ref{teo:triangulacion}). We define a map $\widetilde{\mathbb{L}}:[G,H]\rightarrow [\mathbb{L}(G),\mathbb{L}(H)]$ as follows. Let $f:G\rightarrow H$ be a definable map. We can assume that the underlying sets of both $G$ and $H$ are semialgebraic without parameters. Then, by Fact \ref{t:compa} the map $f$ is definably homotopic to a semialgebraic map $g:G\rightarrow H$ defined without parameters. On the other hand, by Theorem \ref{teo1} there are $\phi_G:\mathbb{L}(G)\rightarrow G(\mathbb{R})$ and $\psi_H:H(\mathbb{R})\rightarrow \mathbb{L}(H)$ definable homotopy equivalences. Finally, we define $\widetilde{\mathbb{L}}([f]):=[\psi_H\circ g(\mathbb{R}) \circ \phi_G]$. Note that $\widetilde{\mathbb{L}}$ is well-defined since it depends neither on the choice of $\phi_G$, $\psi_H$ and $g$ nor the representant of $[f]$.
\end{obs}
\begin{cor}\label{c:catgr}Let $G$ be a definably connected definably compact definable group. Then $\cat(G)^{\mathcal{R}}=\cat(\mathbb{L}(G))$. 
\end{cor}
\begin{proof}By the triangulation theorem we can assume that the underlying  set of $G$ is semialgebraic without parameters. By Theorem \ref{teo1}, $G(\mathbb{R})$ is homotopy equivalent to $\mathbb{L}(G)$ and hence $\cat(G(\mathbb{R}))=\cat(\mathbb{L}(G))$. Then, it follows from Theorem \ref{cor:catdefreal} that $\cat(G)^{\mathcal{R}}=\cat(G(\mathbb{R}))=\cat(\mathbb{L}(G))$, as required. 
\end{proof}
\begin{footnotesize}
\end{footnotesize}
\end{document}